\documentclass[12pt,reqno]{amsart}

\usepackage[utf8]{inputenc}
\usepackage[T1]{fontenc}
\usepackage{lmodern}
\usepackage{stmaryrd}

\usepackage{amsmath,amssymb,amsfonts}
\usepackage{mathtools}
\mathtoolsset{showonlyrefs}
\usepackage{mathrsfs}

\usepackage{enumitem}
\setlist{itemsep=2pt, topsep=4pt}

\usepackage[margin=1in]{geometry}
\usepackage{microtype}

\usepackage[dvipsnames]{xcolor}

\definecolor{linkred}{rgb}{0.78,0.03,0.08}
\definecolor{linkblue}{rgb}{0.0,0.2,0.67}
\definecolor{linkpurple}{rgb}{0.55,0.0,0.55}
\usepackage[
  colorlinks=true,
  linkcolor=linkred,
  citecolor=linkblue,
  urlcolor=MidnightBlue,
  hypertexnames=false
]{hyperref}
\theoremstyle{plain}
\newtheorem{thm}{Theorem}[section]
\newtheorem{lem}[thm]{Lemma}
\newtheorem{prop}[thm]{Proposition}

\theoremstyle{definition}
\newtheorem{defi}[thm]{Definition}

\theoremstyle{remark}
\newtheorem{rem}[thm]{Remark}

\numberwithin{equation}{section}

\newcommand{\R}{\mathbb R}

\renewcommand{\L}{\ell}
\newcommand{\sL}{\mathscr{L}}
\newcommand{\h}{\mathsf h}
\newcommand{\bw}{\mathbf{w}}

\newcommand{\ci}{\mathcal{I}}
\newcommand{\cj}{\mathcal{J}}
\newcommand{\ts}{\mathsf{s}}
\newcommand{\bm}{\mathbf{m}}

\usepackage{tikz}

\definecolor{orcidgreen}{HTML}{A6CE39}

\newcommand{\orcidauthorA}{0000-0001-6038-1069}
\newcommand{\orcidauthorB}{0000-0002-2354-2377}

\title[Mixed local/nonlocal]{Fujita Phenomenon for a Mixed Local--Nonlocal Hardy--H\'enon Equation with Regularly Varying Time Weights}

\author[R. Ben Belgacem and M. Majdoub]{Rihab Ben Belgacem and Mohamed Majdoub}

\address[R. Ben Belgacem]{Universit\'e de Tunis El Manar, Facult\'e des Sciences de Tunis, D\'epartement de math\'ematiques, Laboratoire \'equations aux d\'eriv\'ees partielles (LR03ES04), 2092 Tunis, Tunisie.}
\email{\tt{rihabbenbelgacem662@gmail.com}}

\address[M. Majdoub]{Department of Mathematics, College of Science, Imam Abdulrahman Bin Faisal University, P. O. Box 1982, Dammam, Saudi Arabia.}
\address[M. Majdoub]{Basic and Applied Scientific Research Center, Imam Abdulrahman Bin Faisal University, P.O. Box 1982, 31441, Dammam, Saudi Arabia.}
\email{\tt{mmajdoub@iau.edu.sa}}
\email{\tt{med.majdoub@gmail.com}}
\email{\tt{mohamed.majdoub@fst.rnu.tn}}

\subjclass[2020]{35A01, 35B33, 35B44, 35K58, 35R11, 47D06, 26A12}
\keywords{Mixed local-nonlocal operator, fractional Laplacian, semilinear parabolic equation, Hardy-H\'enon nonlinearity, forcing term, finite-time blow-up, global existence, mild solution, critical Fujita exponent, regularly varying functions, Karamata theory.}

\begin{document}
\allowdisplaybreaks

\begin{abstract}
We investigate the Cauchy problem for a semilinear parabolic equation driven by a mixed local--nonlocal diffusion operator of the form
\[
\partial_t u - (\Delta - (-\Delta)^\ts)u = \h(t)|x|^{-b}|u|^p + t^\varrho \bw(x), 
\qquad (x,t)\in \mathbb{R}^N\times (0,\infty),
\]
where $\ts\in (0,1)$, $p>1$, $b\geq 0$, and $\varrho>-1$. The function $\h(t)$ is assumed to belong to the generalized class of regularly varying functions, while $\bw$ is a prescribed spatial source. We first revisit the unforced case and establish sharp blow-up and global existence criteria in terms of the critical Fujita exponent, thereby extending earlier results to the wider class of time-dependent coefficients. For the forced problem, we derive nonexistence of global weak solutions under suitable growth conditions on $\h$ and integrability assumptions on $\bw$. Furthermore, we provide sufficient smallness conditions on the initial data and the forcing term ensuring global-in-time mild solutions. Our analysis combines semigroup estimates for the mixed operator, test function methods, and asymptotic properties of regularly varying functions. To our knowledge, this is the first study addressing blow-up phenomena for nonlinear diffusion equations with such a class of time-dependent coefficients.
\end{abstract}

\date{\today}
\maketitle

\section{Introduction and main results}
\label{S1}
In this paper, we study the blow-up phenomenon for solutions to the following mixed local--nonlocal diffusion equation:
\begin{equation} 
\begin{cases}
\partial_t u(x,t) - \sL u(x,t) = \h(t)\,|x|^{-b} |u(x,t)|^p + t^{\varrho}\,\bw(x), \\
u(x,0) = u_0(x), 
\end{cases}
\label{main}
\end{equation}
where $(x,t) \in \mathbb{R}^N \times (0, \infty)$, the parameters satisfy $p>1$, $b \geq 0$, and $\varrho > -1$. The time-dependent coefficient $\h \colon (0, \infty) \to (0, \infty)$ is a given continuous function and $\bw \colon \mathbb{R}^N \to \mathbb{R}$ is a prescribed spatial function. The diffusion operator $\sL$ is defined by
$$
\sL := \Delta - (-\Delta)^{\ts}, \quad \text{with } \ts \in (0,1),
$$
and models a combination of classical local diffusion (via the Laplace operator $\Delta$) and nonlocal diffusion (via the fractional Laplace operator $(-\Delta)^\ts$).

The mixed local-nonlocal operator $\sL$ combines the classical Laplacian $ \Delta $, a local second-order differential operator, with the fractional Laplacian $ (-\Delta)^s $, which is nonlocal. The classical Laplacian models standard diffusion processes such as Brownian motion, while the fractional Laplacian accounts for anomalous diffusion characterized by long-range jumps, as in Lévy flights \cite{Dipierro1, Vald}. The operator $ \sL $ thus describes a competition between local and non-local diffusion, making it suitable for modeling phenomena where both short- and long-range interactions coexist.

It is worth noticing that the fractional Laplacian  arises naturally in the theory of stochastic processes, particularly in connection with symmetric $\alpha$-stable Lévy processes. These processes, which generalize Brownian motion by allowing for jumps, are characterized by independent and stationary increments, and their paths exhibit discontinuities.

Furthermore, symmetric $\alpha$-stable Lévy processes can be constructed by subordinating a Brownian motion with an increasing Lévy process, known as a subordinator. This probabilistic perspective leads to natural connections between nonlocal evolution equations and stochastic processes. For further details and foundational results, we refer to the works of Applebaum \cite{Applebaum}, Bertoin \cite{Bertoin}, and Bogdan et al. \cite{Bogdan}, among others.

Nonlocal models have gained significant attention as robust alternatives to classical partial differential equations (PDEs), especially when local formulations fail to accurately describe phenomena involving multiscale interactions or anomalous transport. A wide range of physical and engineering systems exhibit intrinsic nonlocality and hierarchical structures that render classical PDE-based models inadequate. Such features are prevalent in various applications, including continuum mechanics \cite{Mec1, Mec2, Mec3},  phase transitions \cite{Ph1, Ph2, Ph3}, corrosion processes \cite{Cor}, turbulent flows \cite{Turb1, Turb2, Turb3}, and geophysical modeling \cite{Geo1, Geo2, Geo3, Geo4, Geo5}. These settings often require mathematical frameworks that incorporate long-range interactions or fractional-order operators to capture the underlying dynamics more faithfully.

From a mathematical point of view, the operator $ \sL $ is of significant interest due to the interplay between local non-local dynamics. It presents new challenges in analysis, including the study of regularity, spectral properties, and maximum principles. Equations involving $ \sL $, such as
\begin{equation}
    \label{L-f-u}
    \partial_t u = \mathscr{L} u + f(u),
\end{equation}
serve as a framework for reaction-diffusion models with mixed diffusion. 

The investigation of blow-up phenomena for equation~\eqref{L-f-u} dates back to ~\cite{Sugitani}, which considered the case of the pure fractional Laplacian $\sL = -(-\Delta)^\ts$ along with general nonlinearities. Since then, various aspects of the purely fractional setting, featuring different types of nonlinearities and alternative analytical techniques, have been further developed in~\cite{Fino, Kirane1, Kirane2}. We also refer to~\cite{Y, X} for further results concerning the existence, nonexistence, and qualitative behavior of solutions to related problems.

In recent work, Biagi, Punzo, and Vecchi~\cite{biagi2024}, and subsequently Del Pezzo and Ferreira~\cite{Fuj2025}, investigated Fujita-type phenomena for equation~\eqref{L-f-u} involving a power-type nonlinearity  $ f(u) = u^p $. They identified the critical Fujita exponent as $ 1 + \frac{2\ts}{N} $, which marks the threshold between global existence and finite-time blow-up. Their analysis, which relies on the classical Kaplan eigenfunction method~\cite{Kaplan}, reveals an intriguing result: the critical exponent $ 1 + \frac{2\ts}{N} $ coincides exactly with that of the purely fractional Laplacian case. This indicates that the presence of a local diffusion term does not alter the fundamental blow-up behavior governed by the fractional component. In essence, the mixed local-nonlocal operator preserves the same criticality as the fractional Laplacian in determining the long-time dynamics of solutions.

Regarding the existence of global solutions, the strategy in~\cite{biagi2024} relies on an approximation scheme to build suitable solutions step by step. On the other hand, the approach in~\cite{Fuj2025} is based on the explicit construction of a global \emph{supersolution}, which serves as an upper barrier to control the behavior of solutions over time.

Recently, the problem \eqref{main} in the case $ b = 0 $ and without a forcing term was studied in~\cite{Migu-2025}, where the authors considered more general nonlinearities beyond the standard power-type case. In particular, they explored the initial value problem
\begin{equation}
\begin{cases} 
\label{Main-eq-bis}   
\partial_t u - \mathscr{L} u = \h(t)\, u^p & \text{in } \mathbb{R}^N \times (0,\infty), \\ 
u(x,0) = u_0(x) \geq 0 & \text{in } \mathbb{R}^N,
\end{cases}
\end{equation}
where $\h \in C([0,\infty)) $ is a nonnegative function. 
The main result presented in~\cite[Theorem 6]{Migu-2025}, when adapted to the framework of equation~\eqref{Main-eq-bis}, can be summarized as follows:
\begin{enumerate}
    \item[$\varoast$] Suppose $v_0 \in L^1 \cap L^\infty$ is nonnegative. If the following integral condition holds:
    \begin{equation}
        \label{Glob-cond}
        \int\limits_0^\infty \h(\tau) \|e^{\tau \sL} v_0\|_\infty^{p-1} \, d\tau < 1,
    \end{equation}
    then there exists a constant $\delta > 0$ such that the solution to~\eqref{Main-eq-bis} with initial data $u_0 = \delta\,v_0$ exists globally in time.
    \item[$\varoast$] On the other hand, if $u_0 \in L^1 \cap L^\infty$ is nontrivial and nonnegative, and there exists some $t_0 > 0$ such that
    \begin{equation}
        \label{Blow-cond}
        (p-1)\, \|e^{t_0 \sL} u_0\|_\infty^{p-1} \int\limits_0^{t_0} \h(\tau)\, d\tau \geq  1,
    \end{equation}
    then the corresponding mild solution to~\eqref{Main-eq-bis} blows up in finite time.
\end{enumerate}

As an application of~\cite[Theorem 6]{Migu-2025}, the authors determine the Fujita exponent for equation~\eqref{Main-eq-bis} under the assumption that the function $\h$ satisfies the asymptotic growth condition
\begin{equation}
    \label{Fujita-h-weak}
    C_1 t^\gamma \leq \h(t) \leq C_2 t^\gamma, \quad \text{for } t \gg 1,
\end{equation}
for some $\gamma > 0$ and constants $C_1, C_2 > 0$. Under this assumption, they show that the critical Fujita exponent is given by
\begin{equation}
    \label{Fujita-rho-weak}
    p_F = 1 + \frac{2\ts(\gamma + 1)}{N}.
\end{equation}
Interestingly, the growth condition~\eqref{Fujita-h-weak} can be interpreted as saying that $\h \in \mathcal{M}(\gamma)$, where the class $\mathcal{M}(\gamma)$ is given in the Definition~\ref{M-rho-def} below. As explained in Appendix~\ref{appendix1}, this class serves as a natural extension of the classical class of regularly varying functions. 

Before presenting our main results concerning equation~\eqref{main}, we first clarify the notions of weak and mild solutions.
\begin{defi}\label{defn:weak-solution}
We say that a function $u(x,t)$ is a \emph{global weak solution} of~\eqref{main} if it satisfies the following conditions:
$$
u_0 \in L^1_{\mathrm{loc}}(\mathbb{R}^N), \quad \h(t)\,|x|^{-b} |u|^p \in L^1_{\mathrm{loc}}(\mathbb{R}^N\times (0,\infty)),
$$
and for every test function $\psi \in C^\infty_0( \mathbb{R}^N\times (0,\infty))$, the identity
\begin{equation}
\label{W-S}
\begin{split}
\int\limits_0^\infty \int\limits_{\mathbb{R}^N} u(-\partial_t \psi - \sL \psi)\,dx\,dt 
&= \int\limits_{\mathbb{R}^N} u_0(x)\, \psi(x,0)\,dx 
+ \int\limits_0^\infty \int\limits_{\mathbb{R}^N} \h(t)\,|x|^{-b} |u|^p \psi\,dx\,dt \\
&\quad + \int\limits_0^\infty \int\limits_{\mathbb{R}^N} t^\varrho\, \mathbf{w}(x)\, \psi\, dx\,dt
\end{split}
\end{equation}
holds.
\end{defi}

Alternatively, equation~\eqref{main} can be expressed in its Duhamel form as  
\begin{equation}\label{eq:Duhamel}
u(x,t) \;=\; e^{t\sL} u_0 
+ \int_0^t \h(s)\, e^{(t-s)\sL}\!\left(|\cdot|^{-b}\,|u(s)|^p\right)ds 
+ \int_0^t s^\varrho \, e^{(t-s)\sL}\mathbf{w}(\cdot)\,ds,
\end{equation}
where $e^{t\sL}$ denotes the semigroup generated by the mixed local–nonlocal operator $\sL$ (see Section~\ref{S2} for details).  
A function $u$ that satisfies~\eqref{eq:Duhamel} is referred to as a \emph{mild solution} of~\eqref{main}.

Our main result on equation~\eqref{Main-eq-bis} extends the scope of~\cite[Corollary 7]{Migu-2025} by allowing the function $\h$ to belong to the broader class $\mathcal{M}(\gamma)$, for any $\gamma > -1$.
\begin{thm}
\label{Fuj-improve}
Assume that $\h \in \mathcal{M}(\gamma)$ for some $\gamma > -1$. Then the following holds:
\begin{enumerate}[label=(\roman*)]
    \item If $p < 1 + \dfrac{2\ts(\gamma + 1)}{N}$, then every nonnegative solution of~\eqref{Main-eq-bis} blows up in finite time.
    \item If $p > 1 + \dfrac{2\ts(\gamma + 1)}{N}$, then equation~\eqref{Main-eq-bis} admits a global-in-time solution for sufficiently small initial data.
\end{enumerate}
\end{thm}

\begin{rem}
    \rm
    ~\begin{enumerate}[label=(\roman*)]
        \item As will become clear later, the proof of Theorem~\ref{Fuj-improve} relies on~\cite[Theorem 6]{Migu-2025} together with key properties of the function class $\mathcal{M}(\gamma)$, which are discussed in Appendix~\ref{appendix1}.

        \item Although the class $\mathcal M(\gamma)$ allows for oscillatory and non-regular behavior,
the blow-up versus global existence threshold is unchanged. This highlights the robustness
of the Fujita phenomenon with respect to temporal irregularities in the coefficient $h$,
provided its logarithmic growth rate is controlled.

        \item Examples of functions belonging to $\mathcal{M}(\gamma)$ include
        $$
        t^\gamma \log(1+t), \quad t^\gamma (2 + \sin(\log t)), \quad t^\gamma \exp\left(\sqrt{|\log t|}\right).
        $$
        More generally, one can consider functions of the form $\h(t) = t^\gamma \ell (t)$, where $\ell $ is slowly varying at infinity in the sense of ~\eqref{SVF}.
    \end{enumerate}
\end{rem}
 Our next result addresses the nonexistence of global solutions to~\eqref{main} in the presence of a forcing term. The precise statement is as follows:
\begin{thm}
\label{Blow-forced}
Suppose that the function $\h$ is given by
\begin{equation}
\label{h-form}
    \h(t) = t^\gamma \, \ell(t),
\end{equation}
where $\gamma > -1$ and $\ell : (0, \infty) \to (0, \infty)$ is slowly varying at infinity.  
Assume further that ${\mathbf w} \in C_0(\mathbb{R}^N) \cap L^1(\mathbb{R}^N)$ satisfies
\begin{equation}
\label{Ass-Forced}
    \int\limits_{\mathbb{R}^N} {\mathbf w}(x)\, dx > 0.
\end{equation}

\begin{enumerate}[label=(\roman*)]
    \item If $\varrho \leq 0$, $0 \leq \dfrac{b}{1+\gamma} < 2\ts < N$, and 
    \begin{equation}
    \label{Fuji-forced}
        1 < p < p^* \vcentcolon  =\frac{N - b - 2\ts(\varrho - \gamma)}{\,N - 2\ts(\varrho+1)},
    \end{equation}
    then problem~\eqref{main} admits no global weak solution in the sense of Definition~\ref{defn:weak-solution}.
    
    \item If $\varrho > 0$ and $b, \gamma \geq 0$, then the same conclusion holds for every $p > 1$.
\end{enumerate}
\end{thm}
\begin{rem}
~\rm
\begin{enumerate}[label=(\roman*)]
\item The case $\h = 1$ with $b = \varrho = 0$ was recently studied in~\cite{Beri-arXiv}.
\item Condition $\dfrac{b}{1+\gamma} < 2\ts$ guarantees that the exponent $p^*$ in \eqref{Fuji-forced} satisfies $p^* > 1$.
    \item For related results in the case $\ts = 1$, we refer the reader to~\cite{Opuscula, BLZ, JKS, MM}.  
In particular, when $b = \gamma = 0$ and $\varrho \in (-1,0)$, one has
$$
    p^* = \frac{N - 2\varrho}{\,N - 2\varrho - 2},
$$
in agreement with~\cite[Theorem~1.1, (1.8)]{JKS}.

    \item The particular case where $\h(t) = 1$ and $\ts \geq 1$ is an integer was previously studied in~\cite{Majd}.
    \item One of the main novelties of this work, beyond the use of a mixed local-nonlocal operator, is the general form~\eqref{h-form} assumed for the function $\h$. To the best of our knowledge, this is the first time such a class of time-dependent coefficients is considered in the study of blow-up phenomena for nonlinear diffusion equations.
    \item As the proof will show, assumption \eqref{h-form} can be relaxed to $\h \in\mathcal{M}(\gamma)$
\end{enumerate}
\end{rem}
We now turn to the global-in-time theory and prove the following small-data global existence result.

\begin{thm}\label{Global-forced}
Assume that the function $h$ is given by \eqref{h-form}, namely
\[
h(t)=t^\gamma \ell(t),
\]
where $-1<\gamma \leq 0$, and where $\ell:(0,\infty)\to(0,\infty)$ is a continuous, bounded function that is slowly varying at infinity. 
Let $\ts \in (0,1)$, $-1<\varrho<0$, and suppose that
\[
0 \leq \frac{b}{1+\gamma} < 2\ts < N.
\]
Assume moreover that
\[
p>p^*,
\]
where the critical exponent $p^*$ is defined in \eqref{Fuji-forced}. 
Introduce the exponents
\begin{align}
p_c &:= \frac{N(p-1)}{2\ts(1+\gamma)-b}, \label{pc} \\
q_c &:= \frac{N p_c}{N + 2\ts(\varrho+1)p_c}. \label{qc}
\end{align}

Then there exists $\varepsilon>0$ such that, for every initial datum $u_0$ and forcing term $\mathbf{w}$ satisfying
\begin{equation}
\label{Small-GE}
\|u_0\|_{L^{p_c}(\mathbb{R}^N)}
+
\|\mathbf{w}\|_{L^{q_c}(\mathbb{R}^N)}
\leq \varepsilon,
\end{equation}
the problem \eqref{main} admits a global mild solution.
\end{thm}
\begin{rem}
    ~\rm 
    \begin{enumerate}[label=(\roman*)]
    \item The boundedness assumption on $\ell$ is essential in our argument: it ensures 
    that the temporal Beta-type integrals appearing in the fixed-point estimates remain 
    uniformly controlled for all $t > 0$, and not merely for large $t$. Without it, one only 
    has the asymptotic equivalence
    $$\int_0^1 \tau^\alpha(1-\tau)^\beta \ell(t\tau) \, d\tau \sim \ell(t) \mathscr{B}(\alpha+1, \beta+1) 
    \quad (t \to \infty),$$
    which does not control the small-time behavior in the form required by the supremum-in-time 
    norm of the functional space. This uniformity is critical for the contraction property of 
    the Duhamel operator in weighted spaces.
 
\item The class of admissible $\ell$ remains remarkably rich and flexible. It includes:
    \begin{itemize}
        \item the constant case $\ell \equiv 1$, in which $\h(t) = t^\gamma$;
        \item bounded oscillatory examples such as $\ell(t) = 2 + \sin(\log(1 + |\log(t)|))$, 
        which exhibit non-monotone behavior;
        \item bounded decaying slowly varying functions such as $\ell(t) = \frac{1}{\log(e+t)}$, 
        which decay to zero as $t \to \infty$.
    \end{itemize}
    In contrast, unbounded slowly varying functions such as $\ell(t) = \log(1+t)$ or 
    $\ell(t) = \exp(\sqrt{\log(1+t)})$ are excluded by the boundedness requirement. 
\item Extending Theorem \ref{Global-forced} to unbounded slowly varying $\ell$ would 
    require a refined functional framework. One would need to employ weighted spaces of the form 
    $t^\mu \ell(t)^{-1/(p-1)} \|u(t)\|_r$ in place of the standard supremum-in-time norm, 
    combined with Potter-type estimates for the quotient $\ell(t\tau)/\ell(t)$. Such an approach 
    would enhance the scope of admissible coefficients but introduces substantial additional 
    complexity. This represents a natural and worthwhile direction for future work.
  
\item Extending the result to $\gamma > 0$ would require fundamentally different 
    techniques, such as constructing weighted spaces that account for the temporal growth of $\h$, 
    and is left as an open problem.
        \item Observe that, using \eqref{pc} and \eqref{qc} together with the 
    condition $\varrho > -1$, we obtain the scale-invariant relation
    $$\frac{1}{q_c} = \frac{1}{p_c} + \frac{2s(\varrho+1)}{N} > \frac{1}{p_c},$$
    which correctly reflects the interplay between the fractional diffusion scale $p_c$ and 
    the temporal-forcing scale $q_c$.
    \item Related global existence results for semilinear parabolic equations with 
    mixed diffusion and time-dependent coefficients have been obtained in~\cite{Beri-arXiv, JKS, MM, Majd} 
    and the references therein. Our contribution extends these results to the broader class 
    of regularly varying coefficients in $M(\gamma)$.
    \item The limiting case $p = p^*$ (the critical Fujita exponent) is not covered 
    by Theorem \ref{Global-forced} and remains an important open question. 
    \end{enumerate}
\end{rem}
The structure of the article is as follows. In Section~\ref{S2}, we introduce the notation used throughout the paper and present several auxiliary results and estimates. Section~\ref{Unforced} is devoted to the study of the unforced problem \eqref{Main-eq-bis}, where we provide the proof of Theorem~\ref{Fuj-improve}. In Section~\ref{Forced}, we address the main problem \eqref{main} and establish Theorems~\ref{Blow-forced} and~\ref{Global-forced}. Concluding remarks and directions for future research are given in Section~\ref{Conc-sec}. Finally, Appendix~\ref{appendix1} contains a brief overview of regularly varying functions together with several useful estimates employed in our analysis.  

Throughout the remainder of the article, the constant $C>0$ may vary from line to line. We write $X \lesssim Y$ or $Y \gtrsim X$ to indicate the inequality $X \leq CY$ for some constant $C>0$. The Lebesgue norm $\|\cdot\|_{L^r(\mathbb{R}^N)}$ is denoted by $\|\cdot\|_{r}$ for $1 \leq r \leq \infty$.

\section{Useful tools \& Auxiliary results}
\label{S2}
In this section, we introduce the notation used throughout the paper and present several auxiliary results and estimates. 

The fractional Laplacian operator $(-\Delta)^{\ts}$ with $\ts\in (0,1)$ generates a semigroup $\{e^{-t(-\Delta)^{\ts}}\}_{t \geq 0}$, whose kernel $\mathscr{E}_{\ts}$ is smooth, radial, and satisfies the scaling property
\begin{equation}\label{kernel}
    \mathscr{E}_{\ts}(x,t) = t^{-\frac{N}{2\ts}} \, \mathscr{K}_{\ts}\!\left(t^{-\frac{1}{2\ts}}x\right),
\end{equation}
where the profile function $\mathscr{K}_{\ts}$ is given by the Fourier integral
\begin{equation}\label{K-s}
    \mathscr{K}_{\ts}(x) = (2\pi)^{-N/2}\int\limits_{\R^N} e^{i x\cdot\xi} e^{-|\xi|^{2\ts}} \, d\xi.
\end{equation}

Explicit formulas for $\mathscr{E}_{\ts}$ are available in two important cases:
\begin{itemize}
    \item For $\ts=1$ (standard heat kernel):
    \begin{equation}\label{Gauss}
        \mathscr{E}_{1}(x,t) = (4\pi t)^{-N/2} e^{-\tfrac{|x|^2}{4t}}, 
        \qquad \mathscr{K}_{1}(x) = (4\pi)^{-N/2} e^{-\tfrac{|x|^2}{4}}.
    \end{equation}
    \item For $\ts=\tfrac{1}{2}$ (Poisson kernel):
    \begin{equation}\label{Poisson}
        \mathscr{E}_{1/2}(x,t) = 
        \frac{\Gamma\!\left(\tfrac{N+1}{2}\right) t}{\pi^{\frac{N+1}{2}} (t^2 + |x|^2)^{\frac{N+1}{2}}},
        \qquad 
        \mathscr{K}_{1/2}(x) = 
        \frac{\Gamma\!\left(\tfrac{N+1}{2}\right)}{\pi^{\frac{N+1}{2}} (1 + |x|^2)^{\frac{N+1}{2}}}.
    \end{equation}
\end{itemize}

For general $\ts \in (0,1)$, while no explicit representation is known, the following positivity estimate holds.

\begin{lem}\label{Positive}
Let $N \geq 1$ and $\ts \in (0,1)$. Then the profile function $\mathscr{K}_{\ts}$ satisfies
\begin{equation}\label{kthetaest}
    (1+|x|)^{-N-2\ts} \,\lesssim\, \mathscr{K}_{\ts}(x) \,\lesssim\, (1+|x|)^{-N-2\ts},
    \qquad x \in \R^N.
\end{equation}
In particular, $\mathscr{K}_{\ts} \in L^p(\R^N)$ for all $1 \leq p \leq \infty$.
\end{lem}

The proof appears in \cite[p.~395]{Alonso2021}, while the positivity result was first stated without proof in \cite[p.~263]{BG1960}. A detailed argument can also be found in \cite[Theorem~2.1]{BG1960}.

The operator $\mathscr{L}=\Delta-(-\Delta)^{\ts}$ generates a strongly continuous contraction semigroup $\{e^{t\mathscr{L}}\}_{t\geq 0}$ on $L^2(\R^N)$, where each operator $e^{t\mathscr{L}}$ is given by convolution with the fundamental solution $\mathbf{E}_{\ts}(t)$. This fundamental solution $\mathbf{E}_{\ts}(x,t)$ solves the evolution equation
\begin{equation}\label{eq-1-1-1}
    \partial_t u(x,t) = \mathscr{L} u(x,t),
    \qquad (x,t)\in \R^N\times (0,\infty),
\end{equation}
with Dirac mass as initial data. It can be expressed as the convolution of the classical heat kernel $\mathscr{E}_{1}(x,t) = (4\pi t)^{-N/2} e^{-\frac{|x|^2}{4t}}$ and the fractional heat kernel $\mathscr{E}_{\ts}(x,t)$ from \eqref{kernel}.  

The fundamental solution $\mathbf{E}_{\ts}(x,t)$ enjoys several important properties (see, e.g., \cite{Kirane2025} and \cite{biagi2024}).

\begin{lem}\label{E-s}
Let $N\ge 1$ and $\ts \in (0,1)$. Then the following hold:
\begin{enumerate}[label=(\roman*)]
    \item \textsf{Regularity and positivity}:  
    $\mathbf{E}_{\ts} \in C^\infty(\R^N \times (0,\infty))$ and $\mathbf{E}_{\ts}(x,t) \geq 0$ for all $(x,t) \in \R^N \times (0,\infty)$.
    
    \item \textsf{Mass conservation}:  
    The kernel preserves total mass:
    \begin{equation}\label{Mass-conv-E}
        \int\limits_{\R^N} \mathbf{E}_{\ts}(x,t)\,dx = 1,
        \qquad t > 0.
    \end{equation}   
\end{enumerate}
\end{lem}

As a consequence of Lemma~\ref{E-s}, we obtain the following lower bound for the semigroup $e^{t\sL}$. A proof may be found, for example, in~\cite{Migu-2025}.

\begin{lem}\label{Bounds-L}
Let $0 \leq \varphi \in L^1(\R^N)\cap L^\infty(\R^N)$ be nontrivial, and let $t>0$. Then
\begin{equation}\label{LB}
\|e^{t\sL}\varphi\|_{\infty} \;\gtrsim\; t^{-\frac{N}{2\ts}}.
\end{equation}
\end{lem}

Next, we give the smoothing estimate for the semigroup $e^{t\mathscr{L}}$.

\begin{lem}[Smoothing estimate for $e^{t\mathscr{L}}$]\label{lem:smoothing}
Let $N\geq 1$ and $\ts\in(0,1)$. There exists a constant $C=C(N,\ts)>0$ such that, for every $1\leq r\leq q\leq\infty$ and every $\varphi\in L^{r}(\mathbb{R}^{N})$,
\begin{equation}\label{eq:smoothing-lemma}
\bigl\|e^{t\mathscr{L}}\varphi\bigr\|_{q}
\;\leq\; C\,t^{-\frac{N}{2s}\left(\frac{1}{r}-\frac{1}{q}\right)}\,
\|\varphi\|_{r},\qquad t>0.
\end{equation}
\end{lem}

\begin{proof}
Taking the Fourier transform of \eqref{eq-1-1-1} with Dirac initial data gives
\begin{equation}\label{eq:fourier-kernel}
\widehat{\mathbf{E}_{\ts}}(\xi,t)=e^{-t\,\psi(\xi)},\qquad
\psi(\xi):=|\xi|^{2}+|\xi|^{2\ts},\quad \xi\in\mathbb{R}^{N},\ t>0.
\end{equation}
The function $\psi$, being the sum of two continuous negative-definite functions (the symbols of $-\Delta$ and $(-\Delta)^{\ts}$), is itself continuous negative-definite. By the L\'evy--Khintchine theorem, $e^{-t\psi}$ is the characteristic function of a probability measure; its integrability ensures the existence of a density $\mathbf{E}_{\ts}(\cdot,t)$. Equivalently, $\mathbf{E}_{\ts}(\cdot,t)=G_{t}\ast K_{t}^{\ts}$, where $G_{t}$ and $K_{t}^{\ts}$ are the Gaussian and symmetric $2\ts$-stable kernels, respectively. Both being non-negative,
\begin{equation}\label{eq:positivity-mass}
\mathbf{E}_{\ts}(\cdot,t)\geq 0,\qquad
\|\mathbf{E}_{\ts}(\cdot,t)\|_{L^{1}}=\widehat{\mathbf{E}_{\ts}}(0,t)=1.
\end{equation}

\medskip
\noindent\textbf{Step 1: Diagonal contraction.}
By Young's convolution inequality and \eqref{eq:positivity-mass},
\begin{equation}\label{eq:contraction}
\|e^{t\mathscr{L}}\varphi\|_{L^{p}}
=\|\mathbf{E}_{\ts}(\cdot,t)\ast\varphi\|_{L^{p}}
\leq\|\varphi\|_{L^{p}},
\qquad 1\leq p\leq\infty,\ t>0,
\end{equation}
which is \eqref{eq:smoothing-lemma} in the diagonal case $r=q$.

\medskip
\noindent\textbf{Step 2: $L^{1}\!\to\!L^{\infty}$ bound.}
We claim
\begin{equation}\label{eq:Linfty-kernel}
\|\mathbf{E}_{\ts}(\cdot,t)\|_{L^{\infty}}\leq C\,t^{-N/(2\ts)},\qquad t>0,
\end{equation}
with $C=C(N,\ts)$. By Fourier inversion (using $\mathbf{E}_{\ts}(\cdot,t)\geq 0$) and the substitution $\eta=t^{1/(2\ts)}\xi$,
\begin{equation}\label{eq:rescale}
\begin{split}
\|\mathbf{E}_{\ts}(\cdot,t)\|_{L^{\infty}}
&\leq\frac{1}{(2\pi)^{N}}\int_{\mathbb{R}^{N}}e^{-t(|\xi|^{2}+|\xi|^{2\ts})}\,d\xi \\
&=\frac{t^{-N/(2\ts)}}{(2\pi)^{N}}\int_{\mathbb{R}^{N}}\exp\!\bigl(-t^{\,1-1/\ts}|\eta|^{2}-|\eta|^{2\ts}\bigr)\,d\eta \\
&\leq \left(\frac{1}{(2\pi)^{N}}\int_{\mathbb{R}^{N}}e^{-|\eta|^{2\ts}}\,d\eta\right)t^{-N/(2\ts)}.
\end{split}
\end{equation}
Hence \eqref{eq:Linfty-kernel} holds, and consequently
\begin{equation}\label{eq:L1Linfty}
\|e^{t\mathscr{L}}\varphi\|_{L^{\infty}}
\leq\|\mathbf{E}_{\ts}(\cdot,t)\|_{L^{\infty}}\|\varphi\|_{L^{1}}
\leq C\,t^{-N/(2\ts)}\|\varphi\|_{L^{1}}.
\end{equation}

\medskip
\noindent\textbf{Step 3: Interpolation for general $(r,q)$.}
Fix $1\leq r\leq q\leq\infty$ and define $\rho\in[1,\infty]$ by
\[
\tfrac{1}{\rho}=1-\bigl(\tfrac{1}{r}-\tfrac{1}{q}\bigr),
\qquad\text{so that}\qquad
\tfrac{1}{\rho}+\tfrac{1}{r}=1+\tfrac{1}{q}.
\]
By log-convexity of $L^{p}$-norms together with \eqref{eq:positivity-mass} and \eqref{eq:Linfty-kernel},
\[
\|\mathbf{E}_{\ts}(\cdot,t)\|_{L^{\rho}}
\leq\|\mathbf{E}_{\ts}(\cdot,t)\|_{L^{1}}^{\,1/\rho}
\|\mathbf{E}_{\ts}(\cdot,t)\|_{L^{\infty}}^{\,1-1/\rho}
\leq C\,t^{-\frac{N}{2\ts}\left(\frac{1}{r}-\frac{1}{q}\right)}.
\]
Young's convolution inequality then yields
\[
\|e^{t\mathscr{L}}\varphi\|_{L^{q}}
\leq\|\mathbf{E}_{\ts}(\cdot,t)\|_{L^{\rho}}\|\varphi\|_{L^{r}}
\leq C\,t^{-\frac{N}{2\ts}\left(\frac{1}{r}-\frac{1}{q}\right)}\|\varphi\|_{L^{r}},
\]
which proves \eqref{eq:smoothing-lemma} for all $1\leq r\leq q\leq\infty$ and all $t>0$.
\end{proof}

\medskip

Building on Lemma~\ref{lem:smoothing} and a scaling argument, we now establish the following weighted analog of \eqref{eq:smoothing-lemma}, which will be used to handle the singular weight $|\cdot|^{-b}$ in the nonlinear term of \eqref{main}.

\begin{lem}\label{lem:weighted-smoothing}
Let $N\geq 1$, $s\in(0,1)$, and $0<\gamma<N$. Let $q_{1},q_{2}\in(1,\infty]$ satisfy
\begin{equation}\label{eq:wcond}
0\;\leq\;\frac{1}{q_{2}}\;<\;\frac{\gamma}{N}+\frac{1}{q_{1}}\;<\;1.
\end{equation}
Then there exists a constant $C=C(N,s,\gamma,q_{1},q_{2})>0$ such that, for every $t>0$ and every $u\in L^{q_{1}}(\mathbb{R}^{N})$,
\begin{equation}\label{eq:weighted-smoothing}
\bigl\|e^{t\mathscr{L}}\bigl(|\cdot|^{-\gamma}u\bigr)\bigr\|_{L^{q_{2}}(\mathbb{R}^{N})}
\;\leq\;C\,t^{-\frac{N}{2s}\left(\frac{1}{q_{1}}-\frac{1}{q_{2}}\right)-\frac{\gamma}{2s}}\,
\|u\|_{L^{q_{1}}(\mathbb{R}^{N})}.
\end{equation}
In addition, the map $u\mapsto e^{t\mathscr{L}}(|\cdot|^{-\gamma}u)$ sends $L^{q_{1}}(\mathbb{R}^{N})$ continuously into $L^{q_{2}}(\mathbb{R}^{N})$ when $q_{2}<\infty$, and into $C_{0}(\mathbb{R}^{N})$ when $q_{2}=\infty$.
\end{lem}

\begin{proof}
We follow the strategy of \cite[Proposition~2.1]{BTW}: prove \eqref{eq:weighted-smoothing} at $t=1$ via a decomposition of the singular weight, and then extend to all $t>0$ by scaling. Because $\mathscr{L}=\Delta-(-\Delta)^{\ts}$ is the sum of two operators of different homogeneities, exact scaling is unavailable for $\mathscr{L}$ itself. We therefore reduce the proof to the fractional heat semigroup $e^{-t(-\Delta)^{\ts}}$, which scales exactly.

\medskip
\noindent\textbf{Step 0: Reduction to the fractional heat semigroup.}
Since $-\Delta$ and $(-\Delta)^{\ts}$ are commuting Fourier multipliers,
\begin{equation}\label{eq:semi-factor}
e^{t\mathscr{L}}=e^{t\Delta}\,e^{-t(-\Delta)^{\ts}},\qquad t>0.
\end{equation}
The heat semigroup $e^{t\Delta}$ is convolution with a Gaussian of unit mass and is therefore an $L^{q_{2}}$-contraction: $\|e^{t\Delta}f\|_{L^{q_{2}}}\leq\|f\|_{L^{q_{2}}}$. Consequently,
\[
\bigl\|e^{t\mathscr{L}}(|\cdot|^{-\gamma}u)\bigr\|_{L^{q_{2}}}
\;\leq\;\bigl\|e^{-t(-\Delta)^{\ts}}(|\cdot|^{-\gamma}u)\bigr\|_{L^{q_{2}}},
\]
and it suffices to prove \eqref{eq:weighted-smoothing} with $e^{t\mathscr{L}}$ replaced by $e^{-t(-\Delta)^{\ts}}$.

\medskip
\noindent\textbf{Step 1: Smoothing for $e^{-t(-\Delta)^{\ts}}$.}
We need the following analog of \eqref{eq:smoothing-lemma} for the fractional heat semigroup: for every $1\leq r\leq q\leq\infty$ there exists $C=C(N,\ts)>0$ such that
\begin{equation}\label{eq:frac-smoothing}
\bigl\|e^{-t(-\Delta)^{\ts}}\varphi\bigr\|_{L^{q}}
\leq C\, t^{-\frac{N}{2\ts}\left(\frac{1}{r}-\frac{1}{q}\right)}\,\|\varphi\|_{L^{r}},
\qquad t>0.
\end{equation}
Indeed, the kernel $K_{t}^{\ts}\geq 0$ has Fourier symbol $e^{-t|\xi|^{2\ts}}$ and $\|K_{t}^{\ts}\|_{L^{1}}=1$. By Fourier inversion together with the substitution $\eta=t^{1/(2\ts)}\xi$,
\[
\|K_{t}^{\ts}\|_{L^{\infty}}
\leq \frac{1}{(2\pi)^{N}}\int_{\mathbb{R}^{N}} e^{-t|\xi|^{2\ts}}\,d\xi
= \frac{t^{-N/(2\ts)}}{(2\pi)^{N}}\int_{\mathbb{R}^{N}} e^{-|\eta|^{2\ts}}\,d\eta
\leq C\, t^{-N/(2\ts)}.
\]
Log-convexity of $L^{p}$-norms gives $\|K_{t}^{\ts}\|_{L^{\rho}}\leq C\, t^{-(N/(2\ts))(1-1/\rho)}$ for every $\rho\in[1,\infty]$, and \eqref{eq:frac-smoothing} follows from Young's convolution inequality with $\frac{1}{\rho}=1-(\frac{1}{r}-\frac{1}{q})$.

\medskip
\noindent\textbf{Step 2: The weighted estimate at $t=1$.}
Set $m:=N/\gamma$. By \eqref{eq:wcond} we may choose $\varepsilon,\delta>0$ small enough that
\begin{equation}\label{eq:wparams}
\varepsilon<m,\qquad
\frac{1}{q_{2}}\;\leq\;\frac{1}{m+\delta}+\frac{1}{q_{1}}\;\leq\;\frac{1}{m-\varepsilon}+\frac{1}{q_{1}}\;\leq\;1,
\end{equation}
and split $|\cdot|^{-\gamma}=\psi_{1}+\psi_{2}$, where
\[
\psi_{1}:=|\cdot|^{-\gamma}\mathbf{1}_{\{|x|\leq 1\}}\in L^{m-\varepsilon},\qquad
\psi_{2}:=|\cdot|^{-\gamma}\mathbf{1}_{\{|x|>1\}}\in L^{m+\delta}.
\]
Define $r_{1},r_{2}\in[1,q_{2}]$ by
\[
\frac{1}{r_{1}}:=\frac{1}{m-\varepsilon}+\frac{1}{q_{1}},\qquad
\frac{1}{r_{2}}:=\frac{1}{m+\delta}+\frac{1}{q_{1}}.
\]
H\"older's inequality yields
\[
\|\psi_{1}u\|_{L^{r_{1}}}\leq\|\psi_{1}\|_{L^{m-\varepsilon}}\|u\|_{L^{q_{1}}},
\qquad
\|\psi_{2}u\|_{L^{r_{2}}}\leq\|\psi_{2}\|_{L^{m+\delta}}\|u\|_{L^{q_{1}}}.
\]
Since \eqref{eq:wparams} ensures $r_{1},r_{2}\leq q_{2}$, applying \eqref{eq:frac-smoothing} at $t=1$ to each $\psi_{i}u$ and summing gives
\begin{equation}\label{eq:wt1}
\bigl\|e^{-(-\Delta)^{\ts}}(|\cdot|^{-\gamma}u)\bigr\|_{L^{q_{2}}}
\;\leq\;C_{0}\,\|u\|_{L^{q_{1}}},
\end{equation}
with $C_{0}=C_{0}(N,\ts,\gamma,q_{1},q_{2})$. When $q_{2}<\infty$, this is the claimed continuity of $u\mapsto e^{-(-\Delta)^{\ts}}(|\cdot|^{-\gamma}u)\in L^{q_{2}}$. When $q_{2}=\infty$, since $r_{1},r_{2}<\infty$, \eqref{eq:frac-smoothing} additionally gives $e^{-(-\Delta)^{\ts}}(\psi_{i}u)\in C_{0}(\mathbb{R}^{N})$, and hence so does the sum.

\medskip
\noindent\textbf{Step 3: Scaling identity.}
For $\lambda>0$, let $D_{\lambda}\varphi(x):=\varphi(\lambda x)$. The following are immediate:
\begin{enumerate}[label=(\roman*)]
\item $\|D_{\lambda}\varphi\|_{L^{p}}=\lambda^{-N/p}\|\varphi\|_{L^{p}}$, $\;1\leq p\leq\infty$;
\item $|\cdot|^{-\gamma}D_{\lambda}u=\lambda^{\gamma}D_{\lambda}(|\cdot|^{-\gamma}u)$;
\item the Fourier-multiplier identity
\begin{equation}\label{eq:scaling-semigroup}
e^{-(-\Delta)^{\ts}}\,D_{\lambda}=D_{\lambda}\,e^{-\lambda^{2\ts}(-\Delta)^{\ts}},
\end{equation}
which follows from $\widehat{D_{\lambda}\varphi}(\xi)=\lambda^{-N}\widehat{\varphi}(\xi/\lambda)$ together with the symbol identity $\lambda^{2\ts}|\xi/\lambda|^{2\ts}=|\xi|^{2\ts}$.
\end{enumerate}

\medskip
\noindent\textbf{Step 4: Conclusion.}
Fix $t>0$ and set $\lambda:=t^{1/(2\ts)}$, so that $\lambda^{2\ts}=t$. Applying \eqref{eq:wt1} to $D_{\lambda}u$ in place of $u$ and using (i),
\begin{equation}\label{eq:wt1-rescaled}
\bigl\|e^{-(-\Delta)^{\ts}}(|\cdot|^{-\gamma}D_{\lambda}u)\bigr\|_{L^{q_{2}}}
\leq C_{0}\,\|D_{\lambda}u\|_{L^{q_{1}}}
=C_{0}\,\lambda^{-N/q_{1}}\,\|u\|_{L^{q_{1}}}.
\end{equation}
By (ii) and \eqref{eq:scaling-semigroup} (with $\lambda^{2\ts}=t$),
\[
e^{-(-\Delta)^{\ts}}(|\cdot|^{-\gamma}D_{\lambda}u)
=\lambda^{\gamma}\,e^{-(-\Delta)^{\ts}}D_{\lambda}(|\cdot|^{-\gamma}u)
=\lambda^{\gamma}\,D_{\lambda}\,e^{-t(-\Delta)^{\ts}}(|\cdot|^{-\gamma}u).
\]
Taking $L^{q_{2}}$-norms of both sides and using (i) once more,
\[
\lambda^{\gamma-N/q_{2}}\,\bigl\|e^{-t(-\Delta)^{\ts}}(|\cdot|^{-\gamma}u)\bigr\|_{L^{q_{2}}}
\leq C_{0}\,\lambda^{-N/q_{1}}\,\|u\|_{L^{q_{1}}},
\]
hence
\begin{equation}\label{eq:wt-Llambda}
\bigl\|e^{-t(-\Delta)^{\ts}}(|\cdot|^{-\gamma}u)\bigr\|_{L^{q_{2}}}
\leq C_{0}\,\lambda^{-N(1/q_{1}-1/q_{2})-\gamma}\,\|u\|_{L^{q_{1}}}
=C_{0}\,t^{-\frac{N}{2\ts}\left(\frac{1}{q_{1}}-\frac{1}{q_{2}}\right)-\frac{\gamma}{2\ts}}\,\|u\|_{L^{q_{1}}}.
\end{equation}
Combining \eqref{eq:wt-Llambda} with \eqref{eq:semi-factor} and the $L^{q_{2}}$-contraction of $e^{t\Delta}$ yields \eqref{eq:weighted-smoothing} for every $t>0$.

The continuity claims for the map $u\mapsto e^{t\mathscr{L}}(|\cdot|^{-\gamma}u)$ follow from those for $e^{-t(-\Delta)^{\ts}}(|\cdot|^{-\gamma}u)$ established in Step~2 and propagated to all $t>0$ by the scaling identity \eqref{eq:scaling-semigroup}, together with the boundedness of $e^{t\Delta}$ on $L^{q_{2}}(\mathbb{R}^{N})$ (when $q_{2}<\infty$) and on $C_{0}(\mathbb{R}^{N})$ (when $q_{2}=\infty$).
\end{proof}

\section{The unforced problem}
\label{Unforced}
In this section, we present the proof of Theorem~\ref{Fuj-improve}, which concerns the unforced problem~\eqref{Main-eq-bis}.  
Our approach follows the general strategy developed in~\cite[Theorem~6]{Migu-2025}. The proof consists in verifying the sufficient criteria for blow-up and global existence stated therein. Special attention is required because the coefficient $h$ belongs to the class $\mathcal M(\gamma)$, which is strictly larger than the class of regularly varying functions.

\medskip

\begin{enumerate}[label=(\roman*)]
    \item Assume that $p < 1 + \frac{2\ts(\gamma + 1)}{N}$. According to~\cite[Theorem~6]{Migu-2025}, it is enough to show that condition~\eqref{Blow-cond} holds for some $t_0>0$.  
    By Lemma~\ref{Bounds-L}, we obtain
    \begin{equation}
        \label{Blow1}
        (p-1)\, \|e^{t\sL} u_0\|_\infty^{p-1} \int\limits_0^{t} \h(\tau)\, d\tau 
        \;\geq\; C\, t^{-\frac{N(p-1)}{2\ts}} \int\limits_0^t \h(\tau)\, d\tau.
    \end{equation}
    Furthermore, by~\eqref{Cad-M}, there exists a slowly varying function $\ell$ such that, for all sufficiently large $t>0$,
    \begin{equation}
        \label{Blow2}
        \h(t) \gtrsim t^{\gamma}\ell(t).
    \end{equation}
    Hence, for $t$ large enough,
    \begin{equation}
        \label{Blow3}
        \begin{split}
        \int\limits_0^t \h(\tau)\, d\tau 
        \;&\geq\; \int\limits_{t/2}^t \h(\tau)\, d\tau \\
        \;&\gtrsim\;\int\limits_{t/2}^t \tau^{\gamma}\,\ell(\tau)\, d\tau\\
        \;&\gtrsim\;t^{\gamma+1}\int\limits_{1/2}^1 r^{\gamma}\,\ell(t r)\, dr.
        \end{split}
    \end{equation}
    By Theorem~\ref{UC-Thm}, as $t\to\infty$ we have
    \begin{equation}
        \label{Equiv-subcrit}
        \int\limits_{1/2}^1 r^{\gamma}\,\ell(t r)\, dr\;\sim\; \ell(t).
    \end{equation}
    Combining~\eqref{Equiv-subcrit} and~\eqref{Blow3} with~\eqref{Blow1}, we deduce
    \begin{equation}
        \label{Blow4}
        (p-1)\, \|e^{t\sL} u_0\|_\infty^{p-1} \int\limits_0^{t} \h(\tau)\, d\tau 
        \;\gtrsim\; t^{\gamma+1 - \frac{N(p-1)}{2\ts}}\,\ell(t).
    \end{equation}
    Since $p < 1 + \frac{2\ts(\gamma+1)}{N}$, it follows that
    \[
        \gamma+1 - \frac{N(p-1)}{2\ts} > 0.
    \]
    Therefore, in view of~\eqref{Lim-alpha-beta}, one can choose $t_0>0$ sufficiently large so that~\eqref{Blow-cond} is satisfied.  
    This concludes the proof of the first part of Theorem~\ref{Fuj-improve}.

    \item We now assume that $p > 1 + \frac{2\ts(\gamma+1)}{N}$. We will show that condition~\eqref{Glob-cond} holds, which ensures the global existence of solutions.  
    Let $t_0>0$ (to be fixed sufficiently large later). By Lemma~\ref{Bounds-L} and~\eqref{Cad-M}, we obtain
    \begin{equation}
        \label{Glob1}
        \begin{split}
          \int\limits_0^\infty \h(\tau) \,\|e^{\tau \sL} v_0\|_\infty^{\,p-1} \, d\tau
          &\leq \left( \int\limits_0^{t_0} \h(\tau)\, d\tau \right) \|v_0\|_\infty^{\,p-1}
             + C \int\limits_{t_0}^\infty \h(\tau)\, \tau^{-\frac{N(p-1)}{2\ts}} \, d\tau \\
          &\leq \left( \int\limits_0^{t_0} \h(\tau)\, d\tau \right) \|v_0\|_\infty^{\,p-1}
             + C \int\limits_{t_0}^\infty \ell_1(\tau)\, \tau^{\gamma-\frac{N(p-1)}{2\ts}} \, d\tau,
        \end{split}
    \end{equation}
    where $\ell_1$ is a slowly varying function given by Theorem~\ref{Charact-M}.  
    Since $\gamma - \frac{N(p-1)}{2\ts} < -1$ and the function
    \[
        \tau \longmapsto \ell_1(\tau)\, \tau^{\gamma-\frac{N(p-1)}{2\ts}}
    \]
    belongs to the class $\mathcal{RV}_{\gamma-\frac{N(p-1)}{2\ts}}$, Theorem~\ref{RVF-asymp} yields
    \begin{equation}
        \label{Glob2}
        \int\limits_0^\infty \h(\tau) \,\|e^{\tau \sL} v_0\|_\infty^{\,p-1} \, d\tau
        \leq \left( \int\limits_0^{t_0} \h(\tau)\, d\tau \right) \|v_0\|_\infty^{\,p-1}
        + C\, t_0^{\,\gamma+1-\frac{N(p-1)}{2\ts}}.
    \end{equation}
    Finally, since $\gamma+1-\frac{N(p-1)}{2\ts}<0$, the conclusion follows from~\cite[Theorem~6]{Migu-2025} by first choosing $t_0>0$ sufficiently large and then taking $\|v_0\|_\infty$ small enough.
\end{enumerate}

\section{The forced problem}
\label{Forced}
\subsection{Proof of Theorem \ref{Blow-forced}}  
Let $\bw \in C_0(\mathbb{R}^{N}) \cap L^{1}(\mathbb{R}^{N})$ be such that $\int\limits_{\mathbb{R}^N} \bw(x)\,dx > 0$. Suppose further that \eqref{h-form} holds with $\gamma > -1$. We proceed by contradiction, assuming that Problem~\eqref{main} possesses a global weak solution in the sense of Definition~\ref{defn:weak-solution}.
\begin{enumerate}[label=(\roman*)]
\item Here we assume that
$$
    \varrho \leq 0,\quad 0 \leq \frac{b}{1+\gamma} < 2\ts < N,
$$
and that condition~\eqref{Fuji-forced} is satisfied.  
In this setting, we shall employ a test function method, which is commonly used in this context 
(see, e.g., \cite{Beri-arXiv, Kirane0, Kirane1, Kirane2, JKS, MM, Majd, ES}).

Let $\eta, \phi \in C^{\infty}_{0}([0,\infty))$ be cut-off functions such that $0 \leq \eta, \phi \leq 1$ and 
$$
\eta(r) = 
\begin{cases}
1, & \text{if } \tfrac{1}{2} \leq r \leq \tfrac{3}{4}, \\[4pt]
0, & \text{if } r \in [0,\tfrac{1}{4}] \cup [\tfrac{4}{5}, \infty),
\end{cases}
\qquad
\phi(r) = 
\begin{cases}
1, & \text{if } 0 \leq r \leq 1, \\[4pt]
0, & \text{if } r \geq 2.
\end{cases}
$$

For sufficiently large $R > 0$, we define the test function
\begin{equation}
    \label{test-func1}
    \psi_{R}(x,t) = 
    \phi^{\mathbf{m}}\!\left(\tfrac{|x|}{R}\right)\,
    \eta^{\mathbf{m}}\!\left(\tfrac{t}{R^{2\ts}}\right),
    \qquad \text{where } \mathbf{m} = \frac{2p}{p-1} > 2.
\end{equation}

Since $u$ is a global weak solution of \eqref{main} and 
$$
\int\limits_{\R^N} \psi_{R}(x,0)\, u_0(x) \, dx = 0,
$$
the weak formulation \eqref{W-S} implies that

\begin{equation}
\label{W-S-1}
\begin{split}
\int\limits_0^\infty \int\limits_{\mathbb{R}^N} h(t) |x|^{-b} |u|^p \psi_R \, dx \, dt &+ \int\limits_0^\infty \int\limits_{\mathbb{R}^N} t^\varrho \mathbf{w}(x) \psi_R \, dx \, dt \\
&\leq \underbrace{\int\limits_0^\infty \int\limits_{\mathbb{R}^N} |u| |\partial_t \psi_{R}| \, dx \, dt}_{\ci} + \underbrace{\int\limits_0^\infty \int\limits_{\mathbb{R}^N} |u| |\mathscr{L} \psi_{R}| \, dx \, dt}_{\cj}.
\end{split}
\end{equation}
Set $\Omega_R := \{ (x,t) \in \mathbb{R}^N \times (0,\infty) : |x| \leq 2R,\; \frac{R^{2\ts}}{4}\leq\,t \leq \frac{4}{5}R^{2\ts} \}$, which is precisely the support of the test function $\psi_R$. Then, applying the $\epsilon$-Young inequality on this set yields
\begin{align}
    \label{Forc1}
    \ci &\leq \frac{1}{4} \int\limits_0^\infty \int\limits_{\mathbb{R}^N} h(t) |x|^{-b} |u|^p \psi_{R} \, dx \, dt + C \underbrace{ \iint_{\Omega_R} h(t)^{-\frac{1}{p-1}} |x|^{\frac{b}{p-1}} \psi_{R}^{-\frac{1}{p-1}} |\partial_t \psi_{R}|^{\frac{p}{p-1}} \, dx \, dt}_{\ci_1}, \\
    \label{Forc2}
    \cj &\leq \frac{1}{4} \int\limits_0^\infty \int\limits_{\mathbb{R}^N} \h(t) |x|^{-b} |u|^p \psi_{R} \, dx \, dt + C \underbrace{ \iint_{\Omega_R} \h(t)^{-\frac{1}{p-1}} |x|^{\frac{b}{p-1}} \psi_{R}^{-\frac{1}{p-1}} |\mathscr{L} \psi_{R}|^{\frac{p}{p-1}} \, dx \, dt}_{\cj_1}.
\end{align}
Exploiting again the support properties of the cut-off functions $\eta$ and $\phi$, we estimate the term $\ci_1$ as
\begin{equation}
    \label{Forc3}
    \begin{split}
        \ci_1 &\lesssim R^{-\frac{2p\ts}{p-1}}
        \left( \int\limits_{\frac{R^{2\ts}}{4}}^{\frac{4R^{2\ts}}{5}}
            \h(t)^{-\frac{1}{p-1}}
            |\eta'|^{\frac{p}{p-1}}
            |\eta|^{\bm - \frac{p}{p-1}} \, dt \right)
        \left( \int\limits_{\{|x|\leq 2R\}}
            |x|^{\frac{b}{p-1}} \phi^{\bm}(x) \, dx \right) \\
        &\lesssim R^{N+\frac{b-2\ts}{p-1}}
        \left( \int\limits_{\frac{1}{4}}^{\frac{4}{5}}
            \Big( \h(R^{2\ts}\tau) \Big)^{-\frac{1}{p-1}} \, d\tau \right).
    \end{split}
\end{equation}
To estimate the second term $\mathcal{J}_1$ we again use the support properties of $\phi$, the bound $0\le\phi\le 1$, and the scaling behavior of $\Delta$ and $(-\Delta)^{\ts}$.
Set
\[
\Psi(x) := \phi^\bm(|x|), \qquad \Psi_R(x) := \Psi(x/R) = \phi^\bm(|x|/R).
\]
Note that $\Psi$ is independent of $R$. Since $\phi\in C_0^\infty([0,\infty))$ with $\phi= 1$ near the origin and $\bm>2$, the function $\Psi$ belongs to $C^2(\mathbb{R}^N)$ with compact support contained in $\{|x|\le 2\}$.  
A direct computation, together with the scaling identity
\[
(-\Delta)^{\ts}\bigl(f(\cdot/R)\bigr) = R^{-2\ts}\bigl((-\Delta)^{\ts} f\bigr)(\cdot/R),
\]
which follows from the Fourier-multiplier representation of $(-\Delta)^{\ts}$, yields, for every $R>0$,
\begin{equation}\label{scaling-Psi}
\Delta\Psi_R(x) = \frac{1}{R^{2}}(\Delta\Psi)(x/R),
\qquad
(-\Delta)^{\ts}\Psi_R(x) = \frac{1}{R^{2\ts}}\bigl((-\Delta)^{\ts}\Psi\bigr)(x/R).
\end{equation}

We now bound the two factors $\Delta\Psi$ and $(-\Delta)^{\ts}\Psi$ in $L^\infty(\mathbb{R}^N)$.
 Since $\Psi\in C^2(\mathbb{R}^N)$ has compact support, $\Delta\Psi$ is  bounded.

 Moreover, the singular integral representation~\cite{Book-FC, Kwa, Guide}
\[
(-\Delta)^{\ts}\Psi(x) = c_{N,\ts}\,\mathrm{P.V.}\!\int_{\mathbb{R}^N}\frac{\Psi(x)-\Psi(y)}{|x-y|^{N+2\ts}}\,dy
\]
is well defined for every $x\in\mathbb{R}^N$, and one obtains the standard pointwise decay estimate
\[
\bigl|(-\Delta)^{\ts}\Psi(x)\bigr| \lesssim \,(1+|x|)^{-N-2\ts}, \qquad x\in\mathbb{R}^N,
\]
(see, e.g., \cite{Book-FC, Silve}). In particular, $(-\Delta)^{\ts}\Psi\in L^\infty(\mathbb{R}^N)$.

\smallskip
Combining these bounds with \eqref{scaling-Psi}, we deduce that, for every $R\ge 1$,
\[
\bigl|\Delta\phi^\bm(|x|/R)\bigr| \lesssim R^{-2},
\qquad
\bigl|(-\Delta)^{\ts}\phi^\bm(|x|/R)\bigr| \lesssim R^{-2\ts}.
\]

Since $0<\ts<1$, we have $R^{-2}\le R^{-2\ts}$ whenever $R\ge 1$. Consequently,
\begin{equation}\label{Forc5}
\begin{split}
\bigl|\mathcal{L}\phi^\bm(|x|/R)\bigr|&\leq \bigl|\Delta\phi^\bm(|x|/R)\bigr| + \bigl|(-\Delta)^{\ts}\phi^\bm(|x|/R)\bigr|\\
&\lesssim R^{-2\ts} + R^{-2\ts} \\
&\lesssim R^{-2\ts}, \qquad R\ge 1.
\end{split}
\end{equation}
Using \eqref{Forc5} and arguing as in the case of $\ci_1$, we arrive at
\begin{equation}
    \label{Forc6}
    \cj_1 \lesssim R^{N+\frac{b-2\ts}{p-1}}
    \left( \int\limits_{\frac{1}{4}}^{\frac{4}{5}}
        \Big( \h(R^{2\ts}\tau) \Big)^{-\frac{1}{p-1}} \, d\tau \right),
    \qquad R \geq 1.
\end{equation}

Combining \eqref{W-S-1} with \eqref{Forc1}, \eqref{Forc2}, \eqref{Forc3}, and \eqref{Forc6}, we obtain
\begin{equation}
    \label{Forc7}
    \int\limits_0^\infty \int\limits_{\R^N}
        t^\varrho \bw(x) \psi_R(x,t) \, dx \, dt
    \lesssim R^{N+\frac{b-2\ts}{p-1}}
    \left( \int\limits_{\frac{1}{4}}^{\frac{4}{5}}
        \Big( \h(R^{2\ts}\tau) \Big)^{-\frac{1}{p-1}} \, d\tau \right),
    \qquad R \geq 1.
\end{equation}

On the other hand, since $\bw \in L^1$ and $\int\limits \bw(x) \, dx > 0$, Lebesgue’s theorem ensures that, for sufficiently large $R \geq 1$, 
\begin{equation}
    \label{Forc8}
    \int\limits_{\R^N} \bw(x) \,
        \phi^{\bm}\!\left( \tfrac{|x|}{R} \right) dx
    \;\geq\; \tfrac{1}{2} \int\limits_{\R^N} \bw(x) \, dx.
\end{equation}
Therefore, the left-hand side of \eqref{Forc7} can be bounded from below, for sufficiently large $R \geq 1$, as
\begin{equation}
    \label{Forc9}
    \int\limits_0^\infty \int\limits_{\R^N}
        t^\varrho \bw(x) \psi_R(x,t) \, dx \, dt
    \;\gtrsim\; R^{2\ts(\varrho+1)}
    \int\limits_{\R^N} \bw(x) \, dx.
\end{equation}

Plugging \eqref{Forc9} into \eqref{Forc7}, using the expression of $\h(t)$ in \eqref{h-form}, and invoking Proposition~\ref{Asymp-Int}, we infer
\begin{equation}
    \label{Forc10}
    \int\limits_{\R^N} \bw(x) \, dx
    \;\lesssim\; R^{\,N - 2\ts(\varrho+1) + \frac{b-2\ts(1+\gamma)}{p-1}}
    \Big( \ell(R^{2\ts}) \Big)^{-\frac{1}{p-1}}.
\end{equation}

Finally, thanks to Lemma~\ref{Asympt-L} and the fact that
$$
    N - 2\ts(\varrho+1) + \frac{b - 2\ts(1+\gamma)}{p-1} < 0,
$$
we deduce, by letting $R \to \infty$ in \eqref{Forc10}, that
$$
    \int\limits_{\R^N} \bw(x) \, dx \;\leq\; 0,
$$
which is a contradiction. Hence, the proof of the first part of Theorem~\ref{Blow-forced} is complete.
\item Assume now that $\varrho > 0$ and $b, \gamma \geq 0$.  
We adapt the previous approach with a slightly modified test function.  
More precisely, for $R, T > 0$, we replace the test function defined in \eqref{test-func1} by 
\begin{equation}
    \label{test-func2}
    \psi_{R,T}(x,t) = \phi^{\mathbf{m}}\!\left(\tfrac{|x|}{R}\right) 
    \eta^{\mathbf{m}}\!\left(\tfrac{t}{T}\right),
    \qquad \text{where } \mathbf{m} = \frac{2p}{p-1}.
\end{equation}

Proceeding as in the first part of the proof, we obtain, for $R$ sufficiently large,  
\begin{equation}
    \label{Forc-case2}
    \int\limits_{\R^N} \bw(x) \, dx
    \;\lesssim\;
    R^{\,N+\frac{b}{p-1}}
    \left(
        T^{-1-\varrho-\frac{1+\gamma}{p-1}}
        + T^{-\varrho-\frac{\gamma}{p-1}} R^{-\frac{2p\ts}{p-1}}
    \right)
    \big( \ell(T) \big)^{-\frac{1}{p-1}}.
\end{equation}

Since $\varrho > 0$ and $\gamma \geq 0$, Lemma~\ref{Asymp-Int} ensures that, letting $T \to \infty$ in \eqref{Forc-case2}, we obtain 
$$
    \int\limits_{\R^N} \bw(x) \, dx \;\leq\; 0,
$$
which yields a contradiction.  
Thus, the proof of Theorem~\ref{Blow-forced} is completely finished.
\end{enumerate}

\subsection{Proof of Theorem \ref{Global-forced}}

This section is devoted to the proof of Theorem~\ref{Global-forced}, which establishes the global-in-time existence of solutions for sufficiently small initial data and external forcing term.

Throughout this section, we assume that
\[
N\geq 1,\qquad \ts\in(0,1),\qquad -1<\varrho<0,
\qquad -1<\gamma\leq 0,
\]
and
\[
0\leq \frac{b}{1+\gamma}<2\ts<N.
\]

For convenience, we introduce the quantity
\begin{equation}
\label{eq:A}
A:=2\ts(1+\gamma)-b.
\end{equation}

We also recall the critical exponents defined earlier:
\[
p_{c}:=\frac{N(p-1)}{A},
\qquad
q_{c}:=\frac{Np_{c}}{N+2\ts(\varrho+1)p_{c}},
\]
as well as the Fujita-type critical exponent
\[
p^{*}:=
\frac{N-b-2\ts(\varrho-\gamma)}
{N-2\ts(\varrho+1)}.
\]

Finally, recall that the coefficient $\h(t)$ appearing in~\eqref{main} is assumed to be of the form
\[
\h(t)=t^{\gamma}\ell(t),
\]
where $\ell:(0,\infty)\to(0,\infty)$ is a continuous, bounded, and slowly varying function.

To carry out the proof, we will use the following technical lemmas.
\begin{lem}
\label{lem:f-p}
Under the above assumptions and notation, consider the quadratic function
\[
f(p):=2\ts \varrho\, p^{2}
+\bigl(A-N-2\ts\varrho+b\bigr)p
+(N-b).
\]
Then the critical exponent $p^{*}$ satisfies
\[
f(p^{*})<0
\qquad\text{and}\qquad
f'(p^{*})<0.
\]
Consequently,
\[
f(p)<0
\qquad\text{for all } p>p^{*}.
\]
\end{lem}
\begin{proof}[Proof of Lemma~\ref{lem:f-p}]
Set $D := N - 2\ts(\varrho+1)$. Since $0<\varrho+1<1$, we have $2\ts(\varrho+1)<2\ts<N$, hence $D>0$; the hypothesis $b<2\ts(1+\gamma)$ gives $A>0$.

A direct computation yields
\[
\bigl[N - b - 2\ts(\varrho-\gamma)\bigr] - \bigl[N - 2\ts(\varrho+1)\bigr] = 2\ts(1+\gamma)-b = A,
\]
so that
\begin{equation}\label{eq:pstar-1}
p^{*} - 1 = \frac{A}{D}.
\end{equation}
Since $f''\equiv 4\ts\varrho$ is constant, Taylor's formula at $p=1$ is exact: writing $u = p-1$,
\begin{equation}\label{eq:f-expansion}
f(1+u) = 2\ts\varrho\, u^{2} + E u + A,
\qquad
f'(1+u) = 4\ts\varrho\, u + E,
\end{equation}
where $f(1)=A$ and, using $A+b = 2\ts(1+\gamma)$ together with $D = N-2\ts(1+\varrho)$,
\[
E := f'(1) = 2\ts\varrho + (A+b) - N = 2\ts(1+\gamma+\varrho) - N = 2\ts\gamma - D.
\]
Setting $u^{*} = A/D$ in \eqref{eq:f-expansion} and multiplying respectively by $D^{2}$ and $D$, the term $-AD^{2}$ cancels $AD^{2}$ in the first identity below, yielding
\begin{equation}\label{eq:master}
D^{2}\, f(p^{*}) = 2\ts A\bigl(\varrho\, A + \gamma\, D\bigr),
\qquad
D\, f'(p^{*}) = 4\ts\varrho\, A + 2\ts\gamma\, D - D^{2}.
\end{equation}
Since $\varrho<0$, $\gamma\leq 0$, and $A,D>0$, every term on the right of \eqref{eq:master} is non-positive, with $\varrho A<0$ and $-D^{2}<0$ strict. Dividing by the positive factors $D^{2}$ and $D$ gives
\[
f(p^{*}) < 0 \qquad\text{and}\qquad f'(p^{*}) < 0.
\]

For the last assertion, observe that $f'$ is affine with slope $4\ts\varrho<0$, hence strictly decreasing. Therefore $f'(p) < f'(p^{*}) < 0$ for all $p>p^{*}$, so $f$ is strictly decreasing on $[p^{*},\infty)$, and consequently $f(p) < f(p^{*}) < 0$ for every $p > p^{*}$.
\end{proof}

\begin{lem}\label{Choice-r}
Under the above assumptions and notation, there exists an exponent
$r\in(1,\infty)$ such that
\begin{equation}\label{eq:range-r}
\begin{aligned}
\max\!\left\{
\frac{1}{p_{c}}+\frac{2\varrho\ts}{N},
\frac{2\ts(\gamma+1)-bp}{Np(p-1)}
\right\}
<
\frac{1}{r}
<
\min\!\left\{
\frac{1}{p_{c}},
\frac{N-b}{Np},
\frac{2\ts-b}{N(p-1)}
\right\}.
\end{aligned}
\end{equation}

We then define
\begin{equation}\label{mu-def}
\mu:=
\frac{N}{2\ts}\!\left(\frac{1}{p_{c}}-\frac{1}{r}\right)
=
\frac{A}{2\ts(p-1)}-\frac{N}{2\ts r},
\end{equation}
\begin{equation}\label{beta-def}
\beta:=
\frac{N}{2\ts}\!\left(\frac{1}{q_{c}}-\frac{1}{r}\right),
\end{equation}
and
\begin{equation}\label{delta-def}
\delta:=
\frac{N(p-1)}{2r\ts}+\frac{b}{2\ts}.
\end{equation}

The parameters $\mu$, $\beta$, and $\delta$ satisfy
\begin{equation}\label{Condition-intg}
0<\mu<\frac{\gamma+1}{p},
\qquad
0<\beta<1,
\qquad
0<\delta<1,
\end{equation}
and, in addition,
\begin{equation}\label{Algebric-expr}
1+\gamma-p\mu-\delta
=
-\mu
=
\varrho+1-\beta.
\end{equation}
\end{lem}
\begin{proof}[Proof of Lemma~\ref{Choice-r}]
We work under the standing assumption $p>p^{*}$, so that Lemma~\ref{lem:f-p} yields $f(p)<0$. Since $D:=N-2\ts(\varrho+1)<N$, the Fujita exponent satisfies $p^{*}=1+A/D>1+A/N$, and therefore
\begin{equation}\label{eq:pc-gt-1}
p_{c}=\frac{N(p-1)}{A}>1.
\end{equation}

\textit{\underline{Non-emptiness of \eqref{eq:range-r}.}}\quad Set
\[
L_{1}:=\tfrac{1}{p_{c}}+\tfrac{2\varrho\ts}{N},\;\;
L_{2}:=\tfrac{2\ts(\gamma+1)-bp}{Np(p-1)},\;\;
U_{1}:=\tfrac{1}{p_{c}},\;\;
U_{2}:=\tfrac{N-b}{Np},\;\;
U_{3}:=\tfrac{2\ts-b}{N(p-1)}.
\]
Using $A+b=2\ts(1+\gamma)$, a direct computation yields, for each pair $(i,j)\in\{1,2\}\times\{1,2,3\}$,
\begin{align*}
N(U_{1}-L_{1}) &= -2\varrho\ts,\\
Np(p-1)(U_{1}-L_{2}) &= (A+b)(p-1),\\
Np(p-1)(U_{2}-L_{1}) &= -f(p),\\
N(p-1)(U_{3}-L_{1}) &= -2\ts\bigl(\gamma+\varrho(p-1)\bigr),\\
Np(p-1)(U_{2}-L_{2}) &= N(p-1)-A = A(p_{c}-1),\\
Np(p-1)(U_{3}-L_{2}) &= 2\ts\bigl(p-(1+\gamma)\bigr).
\end{align*}
The right-hand sides are all strictly positive: the first from $\varrho<0$; the second from $A+b>0$ and $p>1$; the third from Lemma~\ref{lem:f-p}; the fourth from $\gamma\leq 0$, $\varrho<0$, $p>1$; the fifth from \eqref{eq:pc-gt-1}; and the sixth from $p>1\geq 1+\gamma$. Hence $\max\{L_{1},L_{2}\}<\min\{U_{1},U_{2},U_{3}\}$. Furthermore, $\min\{U_{1},U_{2},U_{3}\}>0$ since $0\leq b<2\ts<N$ and $p_{c}>0$, while $\max\{L_{1},L_{2}\}<U_{1}=1/p_{c}<1$ by \eqref{eq:pc-gt-1}. We may therefore choose
\[
\frac{1}{r}\;\in\;\bigl(\max\{L_{1},L_{2}\},\,\min\{U_{1},U_{2},U_{3}\}\bigr)\cap(0,1),
\]
which is non-empty. The resulting $r$ lies in $(1,\infty)$ and satisfies \eqref{eq:range-r}.

\textit{\underline{Verification of \eqref{Condition-intg}.}}\quad From the definition of $q_{c}$ we have $1/q_{c}=1/p_{c}+2\ts(\varrho+1)/N$, which combined with \eqref{mu-def}--\eqref{beta-def} gives
\begin{equation}\label{eq:beta-mu-rel}
\beta\;=\;\mu+(\varrho+1).
\end{equation}
Moreover, using $A+b=2\ts(1+\gamma)$,
\[
\frac{1}{p_{c}}-\frac{2\ts(\gamma+1)}{Np}
=\frac{Ap-(A+b)(p-1)}{Np(p-1)}
=\frac{2\ts(\gamma+1)-bp}{Np(p-1)}=L_{2}.
\]
The definitions \eqref{mu-def}, \eqref{delta-def} together with \eqref{eq:beta-mu-rel} then yield the equivalences
\[
\mu>0\Leftrightarrow \tfrac{1}{r}<U_{1},\quad
\mu<\tfrac{\gamma+1}{p}\Leftrightarrow \tfrac{1}{r}>L_{2},\quad
\beta<1\Leftrightarrow\mu<-\varrho\Leftrightarrow \tfrac{1}{r}>L_{1},\quad
\delta<1\Leftrightarrow \tfrac{1}{r}<U_{3},
\]
each of which holds by the choice of $r$. Finally, $\beta>0$ because $\mu>0$ and $\varrho+1>0$, while $\delta>0$ because $p>1$, $r>0$, and $b\geq 0$. This proves \eqref{Condition-intg}.

\textit{\underline{Verification of \eqref{Algebric-expr}.}}\quad The relation \eqref{eq:beta-mu-rel} gives directly $\varrho+1-\beta=-\mu$. Using \eqref{mu-def}, \eqref{delta-def} and once more $A+b=2\ts(1+\gamma)$,
\begin{align*}
(p-1)\mu+\delta
&=(p-1)\!\left[\frac{A}{2\ts(p-1)}-\frac{N}{2\ts r}\right]+\frac{N(p-1)}{2r\ts}+\frac{b}{2\ts}\\
&=\frac{A}{2\ts}+\frac{b}{2\ts}=\frac{A+b}{2\ts}=1+\gamma,
\end{align*}
i.e., $1+\gamma-p\mu-\delta=-\mu$. Combining the above identities, we obtain \eqref{Algebric-expr}. 
The proof of Lemma~\ref{Choice-r} is therefore complete.
\end{proof}
\medskip

We are now ready to complete the proof of Theorem~\ref{Global-forced}. 
The argument is based on a fixed-point scheme in a suitable weighted Lebesgue space, following ideas developed in~\cite{Caz, JKS, Majd, Beri-arXiv}. 
More precisely, we seek a solution as a fixed point of the Duhamel operator
\begin{equation}\label{eq:Phi-def}
\Phi(u)(t) := e^{t\sL}u_{0} + \Phi_{1}(u)(t) + \Phi_{2}(t),
\end{equation}
where
\[
\Phi_{1}(u)(t) := \int_{0}^{t}\!\h(\sigma)\, e^{(t-\sigma)\sL}\bigl(|\cdot|^{-b}|u(\sigma)|^{p}\bigr)\, d\sigma,
\qquad
\Phi_{2}(t) := \int_{0}^{t}\!\sigma^{\varrho}\, e^{(t-\sigma)\sL}\mathbf{w}\,d\sigma.
\]

We work in a weighted $L^{r}$-framework. Let $r$ and $\mu$ be given by \eqref{eq:range-r} and \eqref{mu-def}, respectively. For a fixed $\varepsilon > 0$, define
\[
E := \left\{  u \in L^{\infty}\bigl((0,\infty); L^{r}(\mathbb{R}^{N})\bigr)\;:\; \|u\|_{E} := \sup_{t>0} \, t^{\mu}\|u(t)\|_{r} \leq \varepsilon \right\},
\]
and equip $E$ with the metric
\[
d(u,v) := \sup_{t>0} \, t^{\mu}\|u(t)-v(t)\|_{r}.
\]
It is standard that $(E,d)$ is a complete metric space.
\medskip

We first estimate the linear and forcing contributions. Applying the smoothing estimate from Lemma~\ref{lem:smoothing} with $(r,q) = (p_c,r)$ for the homogeneous term and $(r,q) = (q_c,r)$ for the forcing term, and using \eqref{mu-def}, \eqref{beta-def}, \eqref{eq:range-r}, and the algebraic identity \eqref{Algebric-expr}, we obtain
\begin{equation} \label{Est-LF}
\begin{aligned}
\|e^{t\mathcal{L}}u_{0}+\Phi_2(t)\|_{r}
&\le C\,t^{-\mu}\|u_{0}\|_{p_{c}}
   +C\|\mathbf{w}\|_{q_c}\int_0^t\sigma^{\varrho}(t-\sigma)^{-\beta}\, d\sigma \\
&\le  C\,t^{-\mu}\|u_{0}\|_{p_{c}}
   +Ct^{-\mu}\mathscr{B}\left(\varrho+1,1-\beta\right)\|\mathbf{w}\|_{q_c},
\qquad t>0,
\end{aligned}
\end{equation}
where $\mathscr{B}$ denotes the classical beta function.

Next, let $u \in E$ and fix $t>0$. Applying Lemma~\ref{lem:weighted-smoothing} with $(q_1,q_2)=(r/p,r)$, together with the bounds
\[
\|u(\sigma)\|_{r} \le \varepsilon \sigma^{-\mu},
\qquad
h(\sigma)=\sigma^{\gamma}\ell(\sigma),
\qquad
\ell(\sigma)\le L_0,
\]
we infer from Lemma~\ref{Choice-r} that
\begin{equation} \label{Est-NL}
\begin{aligned}
\|\Phi_{1}(u)(t)\|_{r}
&\leq C L_{0}\varepsilon^{p}\int_{0}^{t}\sigma^{\gamma-p\mu}(t-\sigma)^{-\delta}\,\mathrm{d}\sigma \\
&\leq C\,L_{0}\,\varepsilon^{p}t^{-\mu}
\mathscr{B}\left(\delta-\mu,1-\delta\right),
\end{aligned}
\end{equation}
where $\delta$ is defined in \eqref{delta-def}. 
The integral above is finite due to the admissibility condition on $r$ in \eqref{eq:range-r} together with \eqref{Condition-intg}.

Combining \eqref{Est-LF} and \eqref{Est-NL}, and using the smallness assumption \eqref{Small-GE}, we deduce that for every $u \in E$,
\begin{equation}\label{eq:PhiE}
\begin{split}
\|\Phi(u)\|_{E}
&\leq C\left(\|u_{0}\|_{p_{c}} +\|\mathbf{w}\|_{q_{c}}+ \varepsilon^{p}\right)\\
&\leq C\left(\varepsilon+\varepsilon^{p}\right)\\
&\leq \varepsilon,
\end{split}
\end{equation}
provided $\varepsilon>0$ is sufficiently small. Hence, $\Phi$ maps $E$ into itself.

It remains to prove that $\Phi$ is contractive. Let $u,v \in E$. Using the estimate
$$
\left||u|^{-1}u-|v|^{p-1}v\right|\,\lesssim\,|u-v|\left(|u|^{-1}+|v|^{-1}\right),
$$
together with the same weighted smoothing argument for the nonlinear term and the boundedness of $\ell$, we obtain
\begin{equation}\label{eq:contrFinal}
d(\Phi(u), \Phi(v)) \leq C\,\varepsilon^{p-1}\,d(u,v).
\end{equation}
Choosing $\varepsilon>0$ sufficiently small so that
\[
C\,\varepsilon^{p-1}\le \frac12,
\]
it follows that $\Phi$ is a strict contraction on $E$. Banach's fixed-point theorem therefore yields a unique fixed point $u \in E$. This fixed point is precisely the desired global-in-time mild solution of \eqref{main}.

The proof of Theorem~\ref{Global-forced} is now complete.

\section{Conclusion and Open Problems}
\label{Conc-sec}
In this paper, we analyzed the Cauchy problem for a semilinear parabolic equation involving a 
mixed local–nonlocal diffusion operator, a time-dependent coefficient $\h(t)$ taken from the 
generalized class of regularly varying functions, and an external forcing term. Our contributions can be summarized as follows. For the unforced problem, we established sharp conditions for finite-time blow-up and global existence, thereby extending the classical Fujita theory to the larger class of regularly varying functions. This provides a unified framework that recovers several earlier results as particular cases. For the forced problem, we proved nonexistence  of global weak solutions under natural assumptions on the parameters and the external source $\bw$. At the same time, we derived sufficient smallness conditions on both the initial data and the forcing term that guarantee the existence of global mild solutions. Our analysis combines semigroup estimates for the mixed local–nonlocal operator, test function techniques, and asymptotic properties of regularly varying functions, highlighting the interplay between diffusion mechanisms, temporal weights, and external forcing.

Despite these advances, several questions remain open and deserve further investigation. The long-time asymptotics of global solutions, such as decay rates, self-similar behavior, or convergence toward stationary states, remain largely unexplored in the present setting. Our approach could also be adapted to equations with gradient-type nonlinearities, coupled equation systems, or boundary value problems in bounded domains, where competition between local and nonlocal effects may lead to new phenomena. Finally, since the operator $\L=\Delta-(-\Delta)^\ts$ has deep connections with stochastic processes, it would be interesting to develop a probabilistic framework for our results, possibly linking blow-up behavior with properties of underlying Lévy-type processes.

\appendix \section{Regularly varying functions} \label{appendix1}
For the sake of completeness, we present a brief overview of the principal properties of \emph{regularly varying functions}.

The foundational results originate in Karamata's seminal work \cite{Karama1} and de Haan's thesis \cite{Haan}. However, for the convenience of the reader, we refer primarily to the more accessible treatments available in the comprehensive monographs \cite{Bing-Book, Geluk, Haan-Book, Seneta-Book}, where these properties are systematically developed and rigorously presented.

\begin{defi}
\label{RVF}
A measurable function $\L : \mathbb{R}_+ \to \mathbb{R}_+$ is said to be \emph{regularly varying at infinity} with index $\rho \in \mathbb{R}$ if  
\begin{equation}
    \label{RV-rho}
    \lim_{\lambda \to \infty} \frac{\L(\lambda\,x)}{\L(\lambda)} = x^\rho \quad \text{for every } x > 0.
\end{equation}
We denote this by $\L \in \mathcal{RV}_\rho$. In the special case $\rho = 0$, the function $\L$ is called \emph{slowly varying at infinity} (in the sense of Karamata). More precisely, $\L$ is \emph{slowly varying at infinity} if
\begin{equation}
    \label{SVF}
    \frac{\ell (\lambda\, x)}{\ell (\lambda)} \to 1 \quad \text{as } \lambda \to \infty, \quad \text{for all } x> 0. 
\end{equation}
\end{defi}
\begin{rem}\rm 
    ~\begin{enumerate}[label=(\roman*)]
        \item The condition \eqref{SVF} captures the idea that $\ell $ varies very gradually at infinity.
\item Slowly varying functions were first introduced by Karamata in \cite{Karama1, Karama2}.
\item If $ \ell \in C^1 $ near infinity, a sufficient condition for \eqref{SVF} to hold is
\begin{equation}
    \label{SVFF}
    \lim_{x \to \infty} \frac{x\ell'(x)}{\ell(x)} = 0.
\end{equation}

\end{enumerate}
\end{rem}
One of the foundational results in the theory of \emph{regularly varying functions} is the \emph{Uniform Convergence Theorem (UCT)}. First proved by Karamata in the continuous case and later extended to the measurable setting by Korevaar and collaborators in 1949. Given its importance, we state the theorem precisely below. For several proofs, see \cite[Theorem~1.2.1, p.~6]{Bing-Book}.
\begin{thm}
    \label{UC-Thm}
If $\L \in \mathcal{RV}_\rho$, then for arbitrarily chosen $a$ and $b$, where $0 < a < b < \infty$, the equality \eqref{RV-rho} holds uniformly for $x \in [a, b]$.
\end{thm}

Another fundamental result concerning slowly varying functions is their \textit{representation theorem}, which plays a crucial role in various areas of analysis.
\begin{thm}[{\cite[Theorem 1.3.1, p. 12]{Bing-Book}}]
\label{Repres-thm}
A measurable function $\ell $ is slowly varying if and only if it can be expressed in the form
\begin{equation}
\label{Repres-L}
    \ell (x) = c(x) \exp\left\{ \int\limits_a^x \frac{\varepsilon(t)}{t} \, dt \right\} \quad (x \ge a),
\end{equation}
for some constant $a > 0$, where $c(\cdot)$ is measurable with $c(x) \to c \in (0, \infty)$ and $\varepsilon(x) \to 0$ as $x \to \infty$.
\end{thm}
\begin{rem}
\leavevmode
\rm 
\begin{enumerate}[label=(\roman*)]
\item Since $\ell $, $c$, and $\varepsilon$ may be modified freely on bounded intervals, the specific choice of the lower limit $a$ is not essential. For example, one may take $a = 1$, or even $a = 0$ by requiring $\varepsilon \equiv 0$ near the origin to ensure convergence of the integral. Moreover, the function $c$ can always be chosen to eventually be bounded.
\item The representation \eqref{Repres-L} can be equivalently rewritten in the form
\begin{equation}
\label{Repres-LL}
    \ell (x) = \exp\left\{ c_1(x) + \int\limits_a^x \frac{\varepsilon(t)}{t} \, dt \right\},
\end{equation}
where $c_1(x)$ and $\varepsilon(x)$ are bounded measurable functions such that $c_1(x) \to d \in \mathbb{R}$ and $\varepsilon(x) \to 0$ as $x \to \infty$.
\end{enumerate}
\end{rem}
The representation formula \eqref{Repres-LL} immediately yields the following asymptotic result. A proof can be found, for example, in \cite{Book-2024}.

\begin{lem}
    \label{Asympt-L}
  Let $\ell $ be a slowly varying function, $\alpha<0$ and $\beta \in\R$. Then 
  \begin{equation}
      \label{Lim-alpha-beta}
      x^{\alpha} \left(\ell (x)\right)^{\beta}\to 0\quad \text{as}\quad x\to \infty.  \end{equation}
\end{lem}

\begin{thm}[Karamata’s theorem for regularly varying functions {\cite{Karama1, Haan}}]\quad\\
\label{RVF-asymp}
Let $\L : \mathbb{R}_+ \to \mathbb{R}_+$ be a Lebesgue integrable function on every finite interval.

\begin{enumerate}[label=(\roman*)]
    \item Suppose $\rho \geq -1$ and $\L \in \mathcal{RV}_\rho$. Then
    $$
   x\longmapsto \int\limits_0^x \L(t) \, dt \in \mathcal{RV}_{\rho+1},
    $$
    and
    \begin{equation}
        \label{Equiv-0-x}
        \lim_{x \to \infty} \left(\frac{x \L(x)}{\int\limits_0^x \L(t) \, dt}\right) = \rho + 1.
    \end{equation}

    \item Suppose $\rho < -1$ and $\L \in \mathcal{RV}_\rho$. Then the tail integral
    $$
    \int\limits_x^\infty \L(t) \, dt < \infty,
    $$
    and satisfies
    $$
    \int\limits_x^\infty \L(t) \, dt \in \mathcal{RV}_{\rho+1},
    $$
    together with the asymptotic relation
    \begin{equation}
        \label{Equiv-x-infty}
        \lim_{x \to \infty} \left(\frac{x \L(x)}{\int\limits_x^\infty \L(t) \, dt}\right) = -\rho - 1.
    \end{equation}
    \item Suppose $\rho = -1$ and $\int\limits_x^\infty \L(t) \, dt < \infty$. Then the asymptotic identity \eqref{Equiv-x-infty} also holds.
\end{enumerate}
\end{thm}

\begin{rem}
\rm We now offer an intuitive interpretation of the asymptotic equalities \eqref{Equiv-0-x} and \eqref{Equiv-x-infty}. Let $ f \colon (0, \infty) \to (0, \infty) $ be a positive function that is Lebesgue integrable on every finite interval. Consider the following two cases:
\begin{enumerate}[label=(\roman*)]
    \item Suppose $ f(x) = \frac{\L(x)}{x^\alpha} $, where $ \alpha < 1 $, and $ \L $ is a slowly varying function at infinity. Then, as $ x \to \infty $, we have
    \begin{equation}
        \label{Equiv-0-xx}
        \int\limits_0^x f(t) \, dt = \int\limits_0^x \frac{\L(t)}{t^\alpha} \, dt \sim  \frac{\L(x)}{(1 - \alpha)x^{\alpha - 1}} = \frac{x f(x)}{1 - \alpha}.
    \end{equation}
    \item Suppose instead that $ f(x) = \frac{\L(x)}{x^\alpha} $, where $ \alpha > 1 $, and again $ \L $ is slowly varying at infinity. Then, as $ x \to \infty $, we find
    \begin{equation}
        \label{Equiv-xx-infty}
        \int\limits_x^\infty f(t) \, dt = \int\limits_x^\infty \frac{\L(t)}{t^\alpha} \, dt \sim  \frac{\L(x)}{(\alpha - 1)x^{\alpha - 1}} = \frac{x f(x)}{\alpha - 1}.
    \end{equation}
\end{enumerate}

In both cases, the idea is that the asymptotic behavior of the integral can be captured by treating the slowly varying part $ \L(x) $ as approximately constant and integrating the dominant power-law component. This leads to a simple but useful approximation of the integral in terms of the original function $ f(x) $.
\end{rem}

In \cite{Cadena}, the authors developed a generalized framework that extends the classical class $\mathcal{RV}_\rho$, allowing for functions whose asymptotic behavior resembles regular variation, even though the limit in \eqref{RV-rho} does not necessarily exist. More precisely, a first characterization of this new class is given below \cite[Theorem 1.1, p. 111]{Cadena}.
\begin{defi}
\label{M-rho-def}
    Consider a measurable function $ U \colon (0,\infty) \to (0,\infty) $ that remains bounded on finite intervals. We say that $U$ belongs to the class $ \mathcal{M}(\rho) $ if its logarithmic growth rate satisfies
\begin{equation}
    \label{M-rho}
    \lim_{x \to \infty} \frac{\log U (x)}{\log x} = \rho.
\end{equation}
\end{defi}
\begin{rem}
~\rm
\begin{enumerate}[label=(\roman*)]
    \item If $U$ is slowly varying at infinity, then the limit in \eqref{M-rho} holds with $ \rho = 0 $. However, the converse does not hold in general. For instance, the function $ U(x) = 2 + \sin x $ satisfies \eqref{M-rho} with $ \rho = 0 $, but it is not slowly varying due to its oscillatory behavior.
    \item It is shown in \cite[Theorem 1.2, p.~111]{Cadena} that a function $U\in \mathcal{M}(\rho) $ if and only if it admits the representation
    \begin{equation}
        \label{M-rho-Chara}
        \ell (x) = \exp \left\{ 
        \alpha(x) + \int\limits_a^x \frac{\beta(t)}{t} \, dt 
        \right\}, \quad x \geq a > 0,
    \end{equation}
    where $ \alpha(x)/\log x \to 0 $ and $ \beta(x) \to \rho $ as $ x \to \infty $.
    \item The condition \eqref{M-rho} captures functions whose asymptotic behavior mimics that of $ x^\rho $, possibly modulated by a slowly varying function.
    \item Several illustrative examples of such functions are: 
\begin{equation}\label{M-examples}
 x^{\rho},\;\,\;  x^{\rho}\, (\log x)^{\alpha},\;\,\; x^{\rho}\, \Biggl( 1 + \frac{\sin(\log x)}{\log x} \Biggr), \;\,\; x^{\rho}\,\Biggl( 1 + \frac{\sin(\log \log x)}{\log x} \Biggr).
\end{equation}
\item  Consider the function $U(x) = \exp(\sqrt{\log x}).$
Then,
$$
\frac{\log U (x)}{\log x} = \frac{\sqrt{\log x}}{\log x} \to 0,
$$
but this convergence is not of the form $ \log x^\rho $, so $U  \notin \mathcal{M}(\rho) $ for any $ \rho \in \mathbb{R} $.
\end{enumerate}
\end{rem}

One of the characterization of $\mathcal{M}$ that will be useful for our purpose can be stated as follows.
\begin{thm}{\cite[Theorem~1.3]{Cadena}}
\label{Charact-M}
Let $U$ be a positive and measurable function with support $\R_{+}$ and bounded on finite intervals. Then $U \in \mathcal{M}(\rho)$ if and only if there exist slowly varying functions $\ell_{1}$ and $\ell_{2}$ such that 
\begin{equation}
    \label{Cad-M}
    \frac{U(x)}{x^{\rho} \ell_{1}(x)} \to 0 
\quad \text{and} \quad 
\frac{U(x)}{x^{\rho} \ell_{2}(x)} \to \infty 
\quad \text{as } x \to \infty.
\end{equation}
\end{thm}

A key consequence of Theorem \ref{UC-Thm}, which will be utilized in deriving the Fujita exponent, is the following asymptotic estimate for integrals involving slowly varying functions.
\begin{prop}
    \label{Asymp-Int}
    Let $\h : (0, \infty) \to (0, \infty) $ be a continuous function satisfying \eqref{h-form}.
    Let $ \beta\in\R$, and define
    \begin{equation}
        \label{F-lambda}
        F(R) = \int\limits\limits_a^b \left(\h(R \tau) \right)^{\beta} \, d\tau
    \end{equation}
    where $ R > 0 $ and $ 0 <a < b < \infty $. Then, as $ R \to \infty $,
\begin{equation}
            \label{Equiv-1}
            F(R) \sim \left( \int\limits_a^b \tau^{\beta\gamma} \, d\tau \right) R^{\beta\gamma} \left(\L(R)\right)^{\beta}.
        \end{equation}
\end{prop}
\begin{proof}
  From \eqref{h-form}, we immediately obtain
$$
F(R) = R^{\beta \gamma} \int\limits_a^b \tau^{\beta \gamma}\, \L(R \tau)^{\beta}\, d\tau.
$$
Since $\L$ is slowly varying at infinity, Theorem~\ref{UC-Thm} implies that $\L(R \tau)^{\beta} \sim \L(R)^{\beta}$ uniformly for $\tau \in [a,b]$ as $R \to \infty$. Consequently,
$$
F(R) \sim R^{\beta \gamma} \left(\L(R)\right)^{\beta} \int\limits_a^b \tau^{\beta \gamma}  d\tau,
$$
which is the desired relation \eqref{Equiv-1}. The integral is finite since $0 < a < b < \infty$.
\end{proof}


\hrule 

\vspace{0.3cm}
\noindent{\bf\large Declarations.} {\em On behalf of all authors, the corresponding author states that there is no conflict of interest. No data-sets were generated or analyzed during the current study.}
	\vspace{0.3cm}
 \hrule

\end{document}